\documentclass{article}
\usepackage{graphicx} 

\usepackage{xcolor}
\usepackage{amsmath}
\usepackage{amssymb,amsthm}
\usepackage{mathtools}
\usepackage{cite}
\usepackage{framed}
\usepackage{booktabs}
\usepackage{comment}
\usepackage[ruled,vlined]{algorithm2e}
\usepackage{geometry}

\usepackage{mathrsfs}

\usepackage{hyperref}
\usepackage{xurl}

\newcommand{\C}{\mathbb{C}}
\newcommand{\R}{\mathbb{R}}
\newcommand{\cM}{\mathcal{M}}
\newcommand{\cS}{\mathcal{S}}
\newcommand{\cT}{\mathcal{T}}
\newcommand{\sF}{\mathscr{F}}
\newcommand{\sP}{\mathscr{P}}
\newcommand{\Pis}{\Pi_\mathcal{S}}
\newcommand{\target}{\textnormal{target}}

\newcommand{\GE}{G_\varepsilon(E)}

\newcommand{\iu}{\mathrm{i}}
\newcommand{\subjectto}{\ensuremath{\ \text{s.t.}\ }}

\def\eps{\varepsilon}
\DeclarePairedDelimiter{\norm}{\lVert}{\rVert}

\DeclarePairedDelimiter{\scalar}{\langle}{\rangle}
\DeclareMathOperator*{\argmin}{arg\,min}

\DeclareMathOperator*{\rea}{\textnormal{Re}}
\DeclareMathOperator*{\imm}{\textnormal{Im}}
\DeclareMathOperator{\vvec}{vec}

\newtheorem{theorem}{Theorem}
\newtheorem{lemma}[theorem]{Lemma}
\theoremstyle{example}
\newtheorem{example}[theorem]{Example}
\theoremstyle{remark}
\newtheorem{remark}{Remark}

\title{Structured distance to singularity as a nonlinear system of equations}
\author{Miryam Gnazzo\thanks{Istituto di Scienza e Tecnologie dell'Informazione ``Alessandro Faedo'', CNR, Pisa, Italy. \texttt{email: miryam.gnazzo@isti.cnr.it}} \and Nicola Guglielmi\thanks{Division of Mathematics, Gran Sasso Science Institute, L'Aquila, Italy. \texttt{email: nicola.guglielmi@gssi.it} } \and Federico Poloni\thanks{Department of Computer Science, University of Pisa, Pisa , Italy. \texttt{email: federico.poloni@unipi.it}} \and Stefano Sicilia\thanks{Department of Mathematics and Operational Research, University of Mons, Mons, Belgium. \texttt{email: Stefano.SICILIA@umons.ac.be}}}

\begin{document}

\maketitle

\begin{abstract} 
In this article we study the structured distance to singularity for a nonsingular matrix \mbox{$A\in\mathbb{C}^{n\times n}$}, with a prescribed linear structure $\mathcal{S}$ (for instance, a sparsity pattern, or a real Toeplitz structure), i.e., the norm of the smallest perturbation $\Delta \in \mathcal{S}$, such that $A + \Delta$ is singular.
This is an example of structured matrix nearness problem: a family of problems that arise in control and systems theory and in numerical analysis, when characterizing the robustness of a certain property of a system with respect to perturbations that are constrained to a certain structure (for example the structure of the nominal system).
We start by highlighting the parallelism between two main tools which have been proposed in the literature: a gradient system approach for a functional in the eigenvalues, which requires the  solution of certain low-rank matrix differential equations (see [Guglielmi, Lubich, Sicilia, SINUM 2023]), and a two-level optimization approach in which the inner linear least-squares problem is solved explicitly (see [Usevich, Markovsky, JCAM 2014] and [Gnazzo, Noferini, Nyman, Poloni, FoCM 2025]).
In particular, these articles underline the remarkable property that $\Delta$ is (at least generically) the orthogonal projection onto the structure $\mathcal{S}$ of a rank-1 matrix $uv^*$. This property and the parallelism suggest a new reformulation of the problem into a system of nonlinear equations in the two vector unknowns $u,v \in\mathbb{C}^n$. We study this new formulation, and propose an algorithm to solve these nonlinear equations directly with the multivariate Newton's method. We discuss how to avoid the singularity of such system of nonlinear equations, and how to ensure monotonic convergence. The resulting algorithm is faster than the existing ones for large matrices, and maintains comparable accuracy.
\end{abstract}

\medskip

\noindent\textbf{Keywords:} Structured matrix nearness problems, rank-1 perturbations, structured distance to singularity, nonlinear equations

\medskip

\noindent\textbf{AMS subject classifications:} 15A18, 15A99, 65F15

\section{Introduction}

The aim of this work is to describe a new method to compute the nonsingularity radius of a structured matrix, which is a classical matrix nearness problem; for an extensive description of relevant nearness problems, we refer the reader to the seminal paper by N.\,Higham \cite{higham89matrixnearness}. 
For a given nonsingular matrix $A$, we wish to determine the size of the minimal additive perturbation which results in the loss of the considered property, that is, it makes it singular. In the unstructured case such a distance is equal to the minimal singular value of $A$, as is stated by the well-known Eckart-Young-Mirsky theorem. However, when the matrix $A$ has a certain structure, for example it is sparse or Toeplitz, which are cases where the
structure is a linear manifold $\cS$, it appears more interesting to determine the closest singular matrix to $A$ within
the manifold $\cS$. In this case the Eckart-Young-Mirsky theorem is useless, since it provides generically an unstructured perturbation, and the mathematical (as well as the computational) problem is quite challenging. Our aim is to provide ideas and computational techniques to deal with this relevant problem, by proposing a novel algorithm merging ideas from \cite{oracle} and \cite{guglielmi2023rank}. 

For a general nonsingular complex matrix $A \in \mathbb{C}^{n \times n}$, we consider the following problem: find a perturbation $\Delta \in \mathbb{C}^{n \times n}$ of minimal norm such that $A+\Delta$ becomes singular and $\Delta \in \mathcal{S}$, where $\mathcal{S}$ is a linear subspace in $\mathbb{C}^{n\times n}$. 
The minimal norm of such a perturbation is called the \emph{nonsingularity radius} of $A$
\begin{equation}
    \label{eq:distance}
    \min_{\Delta \in \mathcal{S}} \| \Delta \|^2 \subjectto (A+\Delta) \text{ is singular}.
\end{equation}
In this article, the norm $\|\cdot \|$ we consider in \eqref{eq:distance} is the Frobenius norm $\|\cdot \|_F$ i.e., the Euclidean norm of the vector of the matrix entries: for $M \in \mathbb{C}^{n \times n}$, its Frobenius norm is given by
\[
\| M \|_F = 
\Biggl( \sum_{i,j =1}^n |m_{ij}|^2 \Biggr)^{1/2}.
\]

If we do not wish to preserve any structure, the nearest singular matrix to $A$ is obtained directly from its singular value decomposition
\[
A = \sum_{i=1}^n \sigma_i\, u_i v_i^*,
\]
where $u_i, v_i \in \mathbb{C}^n$ are the left and right singular vectors associated with the singular value $\sigma_i > 0$, and~${}^*$ denotes the conjugate transpose.  
A nearest singular matrix is obtained by truncating the last term of the SVD expansion, i.e., by taking
\[
\Delta = - \sigma_n\, u_n v_n^*.
\]
Since the perturbation $\Delta$ is a rank-one matrix, it is a minimizer also for the (operator) $2$-norm.
Thus, the distance to singularity equals the smallest singular value $\sigma_n$, for both the 2-norm and the Frobenius norm.  
This is a special case (rank $r = n-1$) of the classical Eckart--Young theorem~\cite{eckart1936approximation}, originally due to Schmidt~\cite{schmidt1907theorie}.   However, when $A$ belongs to a specified linear structure --- for example, a fixed sparsity pattern, or the class of Toeplitz or Hankel matrices --- 
the situation changes fundamentally: 
a nearest singular matrix within the same structure cannot be directly recovered by truncating the SVD of $A$. A minimal structured perturbation $\Delta$ is no longer rank one in general.  
Nevertheless, for the Frobenius norm we know that a minimal perturbation can be obtained as the orthogonal projection of a rank-one matrix onto the structure
(see, e.g. \cite{GL25}, \cite{guglielmi2023rank}).

This observation motivates the two-level iterative algorithm presented e.g. in \cite{GL25}, \cite{guglielmi2023rank}.  
In the inner iteration we solve a system of differential equations for two vectors that depend on a distance parameter; this requires only matrix--vector multiplications with structured matrices and vector inner products. In the outer iteration we solve a scalar nonlinear equation to determine the structured distance to singularity. The gradient system approach has been proposed also in different contexts where matrix nearness problems arise: in stabilization theory \cite{guglielmi2024stabilization}, for computing the distance of coprime polynomials to common divisibility \cite{guglielmi2017ode}, in neural networks \cite{demarinis2025improving}, in matrix polynomials \cite{gnazzo2025numerical} and in graph theory \cite{andreotti2021measuring,guglielmi2025low}, to name a few.

A different approach to solve the problem comes from the observation that if we impose the stricter constraint that $(A+\Delta)v=0$ for a prescribed vector $v$, the optimal $\Delta_*(v)$ can be found explicitly. Hence, one need only to optimize on $v$. This approach has been studied extensively in a series of papers by Usevich and Markovsky~\cite{MarU14software,UseM14}, even though this idea of ``variable projection'' dates back at least to~\cite{variableprojection}. The approach has recently been revisited and extended in~\cite{oracle}, in particular incorporating a regularization technique. Also this technique can be applied to many of the problems described above, such as stabilization, matrix polynomials, and polynomial GCD.

\subsubsection*{Outline}
The paper is organized as follows. In Section \ref{sec:two methods} we briefly describe two existing methods for the numerical approximation of the structured distance to singularity, namely the ODE--based method and the structured low-rank approximation (SLRA)/Riemann-Oracle method. Moreover, we provide a list of parallels between the two optimization procedures. In Section \ref{sec:reformulation}, we introduce a novel approach for the computation of the structured distance to singularity, recasting the problem as a system of nonlinear equations. In Section \ref{sec:Newton}, we describe the implementation details, and in Section \ref{sec:numerical exp}, we test the behavior of the proposed method via numerical experiments. Finally, Section \ref{sec:conclusions} summarizes the main results.

\section{The two existing methods}
\label{sec:two methods}

\subsection{A gradient system approach}
\label{subsec:gs}
 
 In this paragraph we describe the gradient system approach proposed in many recent works (see e.g. \cite{guglielmi2014computing,guglielmi2023rank,sicilia2025low,GRK17}) for solving matrix nearness problems. We consider its general framework and we focus on how it is applied to the problem of computing the structured distance to singularity. 
 
 Given a matrix $A\in \C^{n\times n}$ and a property $\sP$ related to the spectrum of $A$ we look for the solution of the optimization problem
 \begin{equation}
  \label{prob:dist_stru}
  \argmin_{\Delta \in \cS}\{\|\Delta\|_F : A+\Delta \textnormal{ does not fulfil the property } \sP\},
 \end{equation}
 where $\cS\subseteq \C^{n\times n}$ is a subspace that enforces a structure on the sought perturbation $\Delta$. The gradient system approach relies on a two-level method that splits the original problem into two nested sub-problems that are solved by an \textit{inner iteration} and by an \textit{outer iteration}. In the general setting, we consider a functional $\sF:\cS\rightarrow \R$ such that $\sF(0)>0$ and, for all $\Delta\in\cS$,
 \[
  \sF(\Delta)\leq 0 \iff A+\Delta \textnormal{ does not fulfil the property } \sP, 
 \]
 which means that the functional $\sF$ takes non-positive values if and only if $\Delta$ is admissible. We rewrite the perturbation $\Delta=\eps E$, where $\eps>0$ is the perturbation size and $E$ has unit Frobenius norm and we define the functional $F_\eps$ as
 \[
  F_\eps(E):=\sF(\eps E).
 \]
 The \textit{inner iteration} minimizes the functional $F_\eps$ when the perturbation size $\eps$ is fixed, while the \textit{outer iteration} aims to find the smallest value $\eps_\star$ such that it is possible to have $\sF(\Delta)= 0$ for some perturbation $\Delta\in \cS$ with Frobenius norm $\eps_\star$. The outline of the two-level method is the following: 
 \begin{itemize}
  \item \textit{Inner iteration}: For a fixed $\eps$, compute a matrix perturbation $E_\star(\eps)$ such that
  \begin{equation}
   \label{prob:inn_iter_stru}
   E_\star(\eps)\in\argmin_{\|E\|_F=1, \ E\in \cS} F_\eps(E)=\argmin_{\|\Delta\|_F=\eps, \ \Delta \in \cS} \sF(\Delta),
  \end{equation}
  \item \textit{Outer Iteration}: Find the smallest value $\eps_\star>0$ such that
\begin{equation} \label{eq:zero}  
 \phi(\eps):=F_\eps(E_\star(\eps))=0.
\end{equation}
 \end{itemize}
  
 The \textit{inner iteration} is the most elaborated procedure, while the \textit{outer iteration} is theoretically easier to solve, since it consists of a one-dimensional root-finding problem, even though it could still be challenging. 
 
 In this work we consider the approach of \cite{guglielmi2023rank} used for computing the structured distance to singularity, meaning that the property we aim to violate is $\sP=\{\textnormal{the matrix is nonsingular}\}$ and the functional $\sF$ takes the form
 \begin{equation} \label{eigenvalue optimization problem}
    \sF(\Delta)=|\lambda_{\target}(A+\Delta)|^2,   
 \end{equation}
 where $\lambda_{\target}$ is the eigenvalue with smallest absolute value. We also restrict to consider the case where $\cS$ describes the sparsity pattern of a matrix. To solve problem \eqref{prob:inn_iter_stru}, the \textit{inner iteration} introduces a perturbation matrix path $E(t)$ with $t\geq 0$ and it integrates the matrix ODE
 \begin{equation}
  \label{eq:odeEstru}
  \dot{E}=-\Pis \GE+\rea \langle \Pis \GE,E\rangle E,
 \end{equation}
 where - for two matrices $M,N$ - we denote by 
 \[
 \scalar{M,N} = \operatorname{Tr}(M^*N)
 \]
 the scalar product on matrices that induces the Frobenius norm, by $\GE$ the gradient of $F_\eps$, and by $\Pis$ the orthogonal projection onto $\cS$. The expression of $\GE$ is 
 \begin{equation} \label{GepsExy}
    \GE=-\lambda xy^*,   
 \end{equation}
 where $x$ and $y$ are, respectively, the unit left and right eigenvectors associated with the target eigenvalue $\lambda=\lambda_{\target}(A+\eps E)$ so that $x^*y>0$; the orthogonal projection $\Pis$ simply consists in replacing by $0$ the entries outside of the pattern. 
 
 It is possible to prove that the stationary points of \eqref{eq:odeEstru} corresponds to the local minima of $F_\eps$ and, in order to find them, we integrate the ODE \eqref{eq:odeEstru} until we reach a sought stationary point. Since equation \eqref{eq:odeEstru} is a gradient system, the integration will always end up in a stationary point and, up to non-generic events, these have the form $E\propto \Pis \GE$. 
 
 The matrix $\GE$ has rank-$1$ and hence it follows that the stationary points are the projections onto $\cS$ of a rank-$1$ matrix. This motivates to consider a different ODE to solve the \textit{inner iteration}, whose trajectory belongs to the rank-$1$ manifold $\cM_1$:  introduce a rank-$1$ matrix path $Y(t)\subseteq \cM_1$ such that
 \[
  E(t)=\Pis Y(t), \qquad t\in [0,+\infty)
 \]
 and we consider the ODE
 \begin{equation}
  \label{eq:odeY}
  \dot{Y}=P_Y\left(-G_\eps(\Pis Y)+\rea \langle P_Y(G_\eps(\Pis Y)),\Pis Y \rangle Y\right),
 \end{equation}
 where $P_Y$ denotes the orthogonal projection with respect to the Frobenius inner product onto the tangent space $\cT_Y \cM_1$ of $\cM_1$ in $Y$. There exists an explicit one-to-one correspondence between the stationary points of \eqref{eq:odeEstru} and those of \eqref{eq:odeY} (see \cite[Theorem 3.1]{guglielmi2023rank}), which ensures that we are not introducing nor losing solutions of problem \eqref{prob:dist_stru} if we integrate the low-rank equation instead of the full-rank one. 
\ 
The rank-1 differential equation \eqref{eq:odeY} for $Y=\rho \hat{u} v^*$
can be restated in terms of differential equations for the unit norm vectors $\hat{u}, v$ and an explicit formula for~$\rho$.

\begin{lemma}[Differential equations for the factors]
\label{lem:uv-1-S}
Every  solution $Y(t)\in \cM_1$ of the rank-1 differential equation \eqref{eq:odeY} with $\| \Pis Y(t) \|_F=1$
can be written as $Y(t)=\rho(t)\hat{u}(t)v(t)^*$ where $\hat{u}(t)$ and $v(t)$ of unit norm satisfy the differential equations 
\begin{align*}
\rho \dot {\hat{u}} &=  - \displaystyle \tfrac\iu 2 \imm(\hat{u}^*Gv)\hat{u} -(I-\hat{u}\hat{u}^*) Gv ,
\qquad
\\
\rho \dot v &=   - \displaystyle \tfrac\iu 2 \imm(v^*G\hat{u})v -(I-vv^*) G^*\hat{u} ,
\end{align*}
where $G=G_\eps(E)$ for $E=\Pis Y = \rho\, \Pis (\hat{u}v^*)$
and $\rho=1/\| \Pis(\hat{u}v^*)\|_F$.
\end{lemma}

 Even though equation \eqref{eq:odeY} is not a gradient system, it is somehow close to being one: see \cite[Theorem 4.4]{guglielmi2023rank}. In particular, when choosing a proper starting point sufficiently close to a stationary point, integrating equation \eqref{eq:odeY} always leads to that stationary point. This just yields a local convergence result, weaker than the global convergence property of a gradient system. 

 The first observation is that at a stationary point $Y$ of \eqref{eq:odeY}, we have $P_Y G_\eps(E)=G_\eps(E)$ for $E=\Pis Y$. Therefore, close to a stationary point, 
$P_Y G_\eps(E)$ will be close to $G_\eps(E)$.
It turns out that it is even {\it quadratically} close, as is stated in the following lemma.

\begin{lemma}[Projected gradient near a stationary point] \label{lem:loc-S}
\begin{samepage}
Let $Y_\star\in \cM_1$ with $E_\star=\Pis Y_\star \in \cS$ of unit Frobenius norm. Let $Y_\star$ 
be a stationary point of the rank-1 projected differential equation \eqref{eq:odeY}, 
with an associated target eigenvalue $\lambda$ of $A+\eps E_\star$ that is simple.
Then, there exist $\bar \delta>0$ and a real $C$ such that for all positive $\delta\le \bar\delta$ and all $Y \in \cM_1$ with $\| Y-Y_\star \| \le \delta$ and associated $E=\Pis Y$ of unit norm,
we have
\begin{equation}
\| P_Y G_\eps\left(E \right) - G_\eps\left(E \right) \| \le C \delta^2. 
\end{equation}
\end{samepage}
\end{lemma}

As a direct consequence of this lemma, a comparison of the differential equations \eqref{eq:odeY} and \eqref{eq:odeEstru} yields that when $\delta$-close to a stationary point, the functional decreases monotonically along solutions of \eqref{eq:odeY} up to $O(\delta^2)$, and even with the same negative derivative as for the gradient flow \eqref{eq:odeEstru} up to $O(\delta^2)$. Note that the derivative of the functional is proportional to $-\delta$ in a $\delta$-neighbourhood of a strong local minimum.
Guglielmi, Lubich \& Sicilia \cite{guglielmi2023rank} used Lemma~\ref{lem:loc-S} 
to prove a result on local convergence as $t\to\infty$ to strong local minima of the functional $F_\eps$ for $E(t)=\Pis Y(t)$ of unit Frobenius norm associated with solutions $Y(t)$ of the rank-1 differential equation \eqref{eq:odeY}.

 Numerical experiments show that this is enough to make the method work in practice. In this way, integrating \eqref{eq:odeY} makes it possible to exploit the underlying low-rank features of the matrix nearness problem getting some benefits in the numerical computations.

 \begin{remark}
 \label{rem:sv_version}
In some frameworks, the eigenvalue optimization problem~\eqref{eigenvalue optimization problem} is replaced by the singular value optimization problem 
\[
  \widetilde \sF(\Delta)=\sigma_{\target}(A+\Delta),
\]
where we substitute $\lambda_\target^2$ by $\sigma_\target$ in the definition of $\sF$. This idea is similar to the approach previously employed in~\cite{GLM21}, for instance. In this paper we will consider this specific case, and in particular we will focus on the smallest singular value, that is $\sigma_\target=\sigma_{\min}$, so that 
\[
 \argmin_{\Delta\in \cS} \widetilde{\sF}(A+\Delta)
\]
provides the structured distance to singularity of $A$. In this case it is also possible to study the corresponding gradient system similar to \eqref{eq:odeEstru}, i.e.
\begin{equation}
  \label{eq:odeEstru2}
  \dot{E}=-\Pis \widetilde{G}_\varepsilon(E)+\rea \langle \Pis \widetilde{G}_\varepsilon(E),E\rangle\, E,
 \end{equation}
 where $\widetilde{G}=u v^*$ is the gradient associated with the function $\widetilde{F}_\varepsilon(E):=\widetilde \sF(\varepsilon E)$ (analogous to $F_\varepsilon(E)$ for the eigenvalue case) and $u$ and $v$ are the left and right singular vectors associated with $\sigma_{\min}$
 \end{remark}
\subsection{SLRA / Riemann-Oracle approach} \label{sec:slra}
In this section, we describe a different framework that was studied in a series of papers by Markovsky and Usevich~\cite{UseM14,MarU14software,UseM17gcd}, with the name of \emph{structured low-rank approximation}, and then expanded to more general nearness problems and studied in more detail in~\cite{oracle}, with the name of \emph{Riemann oracle}. For the purpose of this work, we briefly describe the method for the specific case of the computation of the structured distance to singularity for a given nonsingular matrix. We consider again a matrix $A\in \mathbb{C}^{n\times n}$ and a linear subspace $\mathcal{S}$ of dimension $p$ containing the admissible perturbations $\Delta\in \mathbb{C}^{n\times n}$. Let $P^{(1)}, \dots, P^{(p)} \in \mathbb{C}^{n\times n}$ be an orthonormal basis of $\mathcal{S}$, i.e.,
\[
\mathcal{S} = \left\{ \Delta \in \mathbb{C}^{n\times n} \colon \Delta = \sum_{i=1}^p P^{(i)} \delta_i, \quad \delta = \begin{bmatrix}
    \delta_1\\
    \vdots\\
    \delta_p
\end{bmatrix} \in \mathbb{C}^p \right\}, \quad \operatorname{Tr}((P^{(i)})^*P^{(j)}) = \delta_{ij}.
\]
Equivalently, $\vvec(\Delta) = \mathcal{P}\delta$, where
\begin{equation}
\label{eq:matrix_P}
\mathcal{P} =
\begin{bmatrix}
    \vvec P^{(1)} & \dots & \vvec P^{(p)}
\end{bmatrix} \in \mathbb{C}^{n^2\times p}
\end{equation}
has orthonormal columns.

The method is based on the observation that the problem~\eqref{eq:distance} becomes easier if we fix a vector $v\in\mathbb{C}, v\neq 0$ that must be in the kernel of $A+\Delta$ (in the formulation of~\cite{oracle}, this target vector is provided by an \emph{oracle}). To solve
\begin{equation} \label{minfv}
    \min \, \|\Delta\|_F^2 \subjectto {\Delta \in \mathcal{S}}, \, (A+\Delta)v = 0,
\end{equation}
we can eliminate the constraint $\Delta\in\mathcal{S}$ by rewriting~\eqref{minfv} in terms of the vector $\delta$ such that $\vvec \Delta = \mathcal{P}\delta$; then~\eqref{minfv} becomes
\begin{equation} \label{minfvdelta}
    \min\, \norm{\delta}^2 \subjectto M\delta = r,
\end{equation}
where $r = r(v) = -Av\in\mathbb{C}^n$ and
\begin{equation} \label{Mr}
    M = M(v) = (v^\top\otimes I_n)\mathcal{P} = \begin{bmatrix}
        P^{(1)}v & P^{(2)}v & \dots & P^{(p)}v
    \end{bmatrix} \in \mathbb{C}^{n \times p}.
\end{equation}
This is a classical problem whose (unique) solution is $\delta_* = M^+r$, where $M^+$ denotes the Moore-Penrose pseudoinverse (assuming that the feasible region is non-empty, i.e., $r\in \text{range}(M)$).

Hence, we know how to solve the inner minimization subproblem over $\Delta$ in closed form and compute the optimal $\norm{\Delta}_F$ for a given candidate kernel vector $v$, which we can take in the unit sphere in $\mathbb{C}^n$, here denoted by $\mathbb{S}_1$. So to solve the original problem~\eqref{eq:distance} for the Frobenius norm, we solve the outer minimization problem
\[
\min_{v \in \mathbb{S}_1} f(v), \quad f(v) = \norm{M(v)^+r(v)}^2,
\]
via Riemannian optimization over the manifold $\mathbb{S}_1$.

In \cite[Section 3.1]{UseM14}, these formulas are obtained under the additional assumption that $M$ has full column rank. However, as noted in~\cite[Section~2.1]{oracle}, this assumption is problematic: when the rank of $M(v)$ drops, the function $f(v)$ has a removable discontinuity; and the global minimum of $f(v)$ may occur in this discontinuity point. In this case, a classical derivative-based optimization method for $f(v)$ would miss this discontinuity and return instead a suboptimal local minimum.

In \cite{oracle}, this issue is solved using a regularization procedure: we minimize a relaxed version of the objective functional $f(v)$, namely
\begin{equation} \label{fepsilonv}
    f_\eps(v) = \min_{{\Delta\in\mathcal{S}}} \,\|\Delta\|_F^2 + \eps^{-1} \|(A+\Delta)v\|^2,
\end{equation}
for a given $\eps >0$. Arguing as above, one gets
closed-form expressions for the minimum
\[
f_\eps(v) = r^*(MM^* + \eps I)^{-1}r
\]
and for the corresponding argument minimum
\begin{equation} \label{deltastar}
\delta_* = M^*(MM^*+\eps I)^{-1}r = (M^*M+\eps I)^{-1}M^*r, \quad \vvec \Delta_* = \mathcal{P}\delta_*.
\end{equation}
When $\eps \to 0$, the minimum of $f_\eps(v)$ tends to the minimum of $f(v)$~\cite[Theorem~2.11]{oracle}.

Moreover, following the approach used in~\cite{NofNP} for a similar problem, we observe that $\mathcal{P}\mathcal{P}^*$ is the orthogonal projection matrix over $\vvec \mathcal{S}$, hence if we set $u = (MM^*+\eps I)^{-1}r\in\mathbb{C}^n$, we have $\delta_* = M^*u$ and
\begin{equation} \label{vecDelta}
\vvec(\Delta_*) = \mathcal{P}\delta_* = \mathcal{P}M^*u = \mathcal{P}\mathcal{P}^*(\overline{v} \otimes u) = \vvec \Pis(uv^*),    
\end{equation}
i.e., the matrix $\Delta_*$ computed at each iteration is the projection on the structure $\mathcal{S}$ of the rank-1 matrix $uv^*$.

The Euclidean gradient of~\eqref{fepsilonv} is
\[
\nabla_v f_\eps= (A+\Delta)^*u, \quad \Delta = \Pis(uv^*).
\]
In a stationary point $v_*$ of $f_\eps$ on the unit sphere, the Euclidean gradient $\nabla_{v_*} f$ must be a multiple of $v_*$. In the limit $\eps \to 0$ the function $f(v)$ is scale-invariant, i.e., $f(v\alpha)=f(v)$ for each $\alpha\in\mathbb{C},v\in\mathbb{C}^n$, so the radial component of the gradient vanishes, and it must be the case that $\nabla_{v_*} f = 0$.

As mentioned, in this work we are mainly interested in sparsity structures. In this case, $MM^*$ becomes a diagonal matrix (see~\cite[Section~5]{oracle} and the proof of Lemma~\ref{lem:Hessian} below), making the method simpler and faster.

\subsection{Similarities between the two methods}

It is interesting to note the parallels between the two methods.

\begin{center}
    \begin{tabular}{p{0.48\textwidth} p{0.48\textwidth}}
    \toprule
    \textbf{ODE approach} & \textbf{Riemann--Oracle approach} \\
    \midrule
    Inner iteration: for given $\eps>0$, use an ODE integrator to compute
    $\Delta_* = \Pi_\mathcal{S}(uv^*)$, with $u=\eps \rho \hat{u}$, which satisfies $\dot{Y}=0$.
    & 
    Inner iteration: for given $\eps > 0 $, use Riemannian optimization to compute \mbox{$\Delta_* = \Pi_{\mathcal{S}}(uv^*)$} which satisfies \mbox{$\nabla_v f_{\eps} = v\mu$}. \\
    \midrule
    Outer iteration: compute optimal $\eps$ by Newton-bisection to obtain $\lambda = 0$ (univariate minimization problem). &
    Outer iteration: let $\eps \to 0$.\\
    \midrule 
    At each inner iteration: \mbox{$(A+\Delta_*)v=v\lambda$}, \mbox{$(A+\Delta_*)^*u = u \overline{\lambda}$}.
    &
    At each inner iteration: $\|\Delta\|_F^2 + \eps^{-1} \|(A+\Delta)v\|^2$ is minimized and $(A+\Delta_*)^*u = v\mu$.\\
    \midrule
    \multicolumn{2}{c}{
    At convergence of the outer iteration: $(A+\Delta_*)v=0$, $u^*(A+\Delta_*) = 0$.}
    \\
    \bottomrule
    \end{tabular}
\end{center}
While the two methods are not the same, their structure is very similar: they minimize different functions at each step, but in the end of their nested iterations they compute a pair $(u,v)$ which satisfies the same nonlinear system of equations.

\section{Reformulation as a nonlinear system}
\label{sec:reformulation}

These two formulations suggest looking for vectors $u,v\in\mathbb{C}^n$ that satisfy the two required equations directly. 

\begin{framed}
\noindent Problem: find $u,v\in\mathbb{C}^n$ that satisfy the non-linear system of equations
\begin{equation} \label{nlsys}
    (A + \Pis(uv^*))v = (A + \Pis(uv^*))^* u = 0.    
\end{equation}    
\end{framed}
When these two equations~\eqref{nlsys} hold, the pair $(u,v)$ is a critical point of the optimization problems in both the ODE and the Oracle approach. 

We prove the following statement:
the vectors $u,v$ computed by the ODE approach are, up to a rescaling, a solution of Problem \eqref{nlsys}.

\begin{theorem}[Stationary points for fixed $\varepsilon$]
\label{thm:stationary-fixed-eps}
Let $\widetilde{\sF}(\Delta)=\sigma_{\min}(A+\Delta)$ and fix $\varepsilon>0$.
Assume that $\sigma_{\min}(A+\varepsilon E)$ is positive and simple along the
trajectory considered by the ODE approach. Then any stationary point
$E(\varepsilon)\in \mathcal{S}$ of the constrained flow with $\|E(\varepsilon)\|_F=1$
satisfies, up to multiplication by a real scalar,
\begin{equation}
E(\varepsilon) \propto \Pis\!\bigl(u(\varepsilon)v(\varepsilon)^*\bigr),
\label{eq:rank1proj-stationary}
\end{equation}
where $u(\varepsilon),v(\varepsilon)$ are the left and right singular vectors of unit norm, associated with $\sigma(\varepsilon):=\sigma_{\min}(A+\varepsilon E(\varepsilon))$. In particular, with the normalization
\begin{equation}
E(\varepsilon)
= \frac{\Pis(u(\varepsilon)v(\varepsilon)^*)}
{\|\Pis(u(\varepsilon)v(\varepsilon)^*)\|_F},
\label{eq:E-structure-thm1}
\end{equation}
the singular-vector relations read
\begin{equation}
(A+\varepsilon E(\varepsilon))v(\varepsilon)=\sigma(\varepsilon)u(\varepsilon),
\qquad
(A+\varepsilon E(\varepsilon))^*u(\varepsilon)=\sigma(\varepsilon)v(\varepsilon).
\label{eq:sv-rel-thm1}
\end{equation}
\end{theorem}

\begin{proof}
Let $\Delta(t)=\varepsilon E(t)$ with $\|E(t)\|_F\equiv 1$. As in Remark \ref{rem:sv_version}, we define the functional $\widetilde{F}_\varepsilon(E)=\widetilde{\sF}(\varepsilon E)=\sigma_{\min}(A+\varepsilon E)$. Whenever $\sigma_{\min}$ is simple and strictly positive one has the standard differential identity
\[
\frac{\mathrm{d}}{\mathrm{d}t}\widetilde{F}_\varepsilon(E(t))
= \varepsilon\,\rea\langle \dot E(t),\,u(t)v(t)^*\rangle,
\]
where $u(t),v(t)$ are the corresponding left/right singular vectors.
Hence the (Euclidean) gradient of $\widetilde{F}_\varepsilon$ at $E$ is $\widetilde{G}_\varepsilon(E)=u v^*$.

The constrained gradient flow on the Frobenius-norm unit sphere in $S$ is
\[
\dot{E}=-\Pis \widetilde{G}_\varepsilon(E)+\rea\langle \Pis \widetilde{G}_\varepsilon(E),E\rangle\,E,
\]
so at a stationary point $E(\varepsilon)$ we have
$-\Pis \widetilde{G}_\varepsilon(E(\varepsilon))+\rea\langle \Pis \widetilde{G}_\varepsilon(E(\varepsilon)),
E(\varepsilon)\rangle E(\varepsilon)=0$,
which implies \eqref{eq:rank1proj-stationary}. Choosing the normalization
\eqref{eq:E-structure-thm1} yields \eqref{eq:sv-rel-thm1}.
\end{proof}
Note that we had to assume $\sigma_{\min}>0$ in the previous theorem. In the next theorem we extend to the limit to obtain a result for $\sigma_{\min}=0$.
\begin{theorem}[Limit characterization at $\varepsilon_\star$] \label{th:limit}
Assume that for every $\varepsilon \in (0,\varepsilon_\star)$ there exists a
stationary point $E(\varepsilon)\in \mathcal{S}$ of unit Frobenius norm,
$\|E(\varepsilon)\|_F = 1$, such that the smallest singular value
\[
\sigma(\varepsilon)
:= \sigma_{\min}(A+\varepsilon E(\varepsilon))
\]
is positive and simple, and the associated left and right singular vectors
$u(\varepsilon), v(\varepsilon)$ (of unit norm) satisfy
\begin{equation}
(A+\varepsilon E(\varepsilon)) v(\varepsilon)
= \sigma(\varepsilon) u(\varepsilon),
\qquad
(A+\varepsilon E(\varepsilon))^* u(\varepsilon)
= \sigma(\varepsilon) v(\varepsilon),
\label{eq:sv-rel}
\end{equation}
and
\begin{equation}
E(\varepsilon)
= \frac{\Pis(u(\varepsilon)v(\varepsilon)^*)}
{\|\Pis(u(\varepsilon)v(\varepsilon)^*)\|_F}.
\label{eq:E-structure}
\end{equation}
Moreover suppose that
\[
\lim_{\varepsilon \nearrow \varepsilon_\star}
\sigma(\varepsilon) = 0,
\]
and that there exists a constant $c>0$ such that
\begin{equation}
\|\Pis(u(\varepsilon)v(\varepsilon)^*)\|_F \ge c
\qquad \text{for all } \varepsilon\in(0,\varepsilon_\star).
\label{eq:nondeg}
\end{equation}
Then there exist vectors $\widehat{u}, \widehat{v} \in \mathbb{C}^n$,
with $\|\widehat{v}\|_2=1$, and a matrix
$
\widehat{\Delta} := \Pis(\widehat{u}{\widehat{v}}^*) 
$
such that
\begin{equation}
(A+\widehat{\Delta}) \widehat{v} = 0, \qquad \mbox{and} \qquad
(A+\widehat{\Delta})^* \widehat{u} = 0,
\label{eq:limit-system}
\end{equation}
i.e., $(\widehat{u},\widehat{v})$ satisfies the nonlinear system
$
(A+\Pis(\widehat{u}{\widehat{v}}^*))\widehat{v} =
(A+\Pis(\widehat{u}{\widehat{v}}^*))^*\widehat{u} = 0.
$
\end{theorem}

\begin{proof}
Define
$\Delta(\varepsilon) := \varepsilon E(\varepsilon)$.
By \eqref{eq:E-structure} we can rewrite
\[
\Delta(\varepsilon)
=
\Pis\!\left(
\alpha(\varepsilon)\, u(\varepsilon)v(\varepsilon)^*
\right),
\qquad \mbox{where}
\qquad
\alpha(\varepsilon)
:=
\frac{\varepsilon}
{\|\Pis(u(\varepsilon)v(\varepsilon)^*)\|_F}.
\]
Set $\tilde u(\varepsilon) := \alpha(\varepsilon)\, u(\varepsilon)$.
Then
\begin{equation}
\Delta(\varepsilon)
=
\Pis(\tilde u(\varepsilon)v(\varepsilon)^*).
\label{eq:Delta-rewrite}
\end{equation}
By \eqref{eq:nondeg}, the scalars $\alpha(\varepsilon)$ are bounded on
$(0,\varepsilon_\star)$, hence $\tilde u(\varepsilon)$ is bounded as well,
since $\|u(\varepsilon)\|_2=1$.
Moreover, $\|v(\varepsilon)\|_2=1$ for all $\varepsilon$.
Therefore, by compactness of the unit sphere, there exists a sequence
$\varepsilon_k \nearrow \varepsilon_\star$ such that
\[
v(\varepsilon_k) \to \widehat{v}, \qquad
\tilde u(\varepsilon_k) \to \widehat{u}
\]
for some $\widehat{u}, \widehat{v}$ with $\|\widehat{v}\|_2=1$.
Since $\|\Delta(\varepsilon)\|_F=\varepsilon$, the sequence
$\Delta(\varepsilon_k)$ is bounded and hence (up to subsequences)
$\Delta(\varepsilon_k) \to \widehat{\Delta} \in \mathcal{S}.$
Passing to the limit in \eqref{eq:Delta-rewrite} and using continuity
of $\Pis$ yields
\begin{equation} \label{eq:Deltahat}
\widehat{\Delta} = \Pis(\widehat{u}{\widehat{v}}^*).
\end{equation}

Now consider the first relation in \eqref{eq:sv-rel}:
$(A+\Delta(\varepsilon_k)) v(\varepsilon_k)
= \sigma(\varepsilon_k) u(\varepsilon_k)$.
Taking norms gives
\[
\|(A+\Delta(\varepsilon_k)) v(\varepsilon_k)\|_2
= \sigma(\varepsilon_k),
\]
and by assumption $\sigma(\varepsilon_k)\to 0$.
Since 
\[
A+\Delta(\varepsilon_k) \to A+\widehat{\Delta},
\qquad
v(\varepsilon_k) \to \widehat{v},
\]
we obtain $(A+\widehat{\Delta}) \widehat{v} = 0$.

For the adjoint relation, multiply the second equation
in \eqref{eq:sv-rel} by $\alpha(\varepsilon_k)$:
\[
(A+\Delta(\varepsilon_k))^* \tilde u(\varepsilon_k)
=
\alpha(\varepsilon_k)\sigma(\varepsilon_k)
\, v(\varepsilon_k).
\]
The right-hand side tends to zero because
$\alpha(\varepsilon_k)$ is bounded and
$\sigma(\varepsilon_k)\to 0$.
Passing to the limit yields
$(A+\widehat{\Delta})^* \widehat{u} = 0$.
This proves \eqref{eq:limit-system}.
\end{proof}

The proof can be simplified if we assume that $\exists \lim\limits_{\eps \nearrow \eps_\star} u(\eps) = \widehat{u}$ and $\lim\limits_{\eps \nearrow \eps_\star} v(\eps) = \widehat{v}$.

We can also prove an explicit stationarity property of the limit matrix.

\begin{theorem}[Clarke-stationarity at $\widehat\Delta$]
\label{thm:clarke}
Under the assumptions of Theorem \ref{th:limit}, let 
\[
\cS_{\varepsilon_\star}:=\{\Delta\in \cS:\ \|\Delta\|_F=\varepsilon_\star\},
\qquad
T_{\widehat\Delta}\cS_{\varepsilon_\star}
=\{Z\in \cS:\ \rea\langle Z,\widehat\Delta\rangle=0\}
\]
and
\[
\widetilde{\sF}(\Delta):=\sigma_{\min}(A+\Delta),\qquad \Delta\in \cS.
\]

Then $\widehat \Delta$ is Clarke stationary (see \cite{clarke1975generalized} for more details) on $\cS_{\varepsilon_\star}$.
Equivalently, with $\Pi$ the projector onto the tangent space,
\[
0\in \Pi_{T_{\widehat\Delta}\cS_{\varepsilon_\star}}
\bigl(\partial^C \widetilde{\sF}(\widehat\Delta)\bigr),
\]
i.e.\ there exists $G\in \partial^C\sigma_{\min}(A+\widehat\Delta)$ (the Clarke subgradient)
such that
$\Pis(G)$ is collinear with $\widehat\Delta$.
\end{theorem}

\begin{proof}
Using the same arguments of Theorem \ref{th:limit}, we let $\Delta_k = \eps_k E_k$.
Since $\sigma_k = \sigma_{\min}(A + \Delta_k) >0$ is simple, $X\mapsto \sigma_{\min}(X)$ is differentiable at
$X_k:=A+\Delta_k$ and its gradient is $u_k v_k^*$. 
Stationarity of $E_k$ on the unit sphere implies the gradient is collinear with
$E_k$.

Since $\|u_k v_k^*\|_F=1$, the sequence $\{u_k v_k^*\}$ is bounded; consider a
convergent subsequence such that $u_k v_k^*\to G$.
We obtain
\[
\Pis(u_k v_k^*)=\|\Pis(u_k v_k^*)\|_F\,E_k
\ \Longrightarrow\
\Pis(G)=\eta\,\widehat E,
\qquad
\widehat E:=\widehat\Delta/\varepsilon_\star,
\]
for some $\eta > 0$ (by the nondegeneracy assumption \eqref{eq:nondeg}), hence $\Pis(G)$ is collinear with $\widehat\Delta$.
The map $X\mapsto \sigma_{\min}(X)$ is locally Lipschitz, and its Clarke
subdifferential is outer semicontinuous. 
Thus, since $X_k\to \widehat X:=A+\widehat\Delta$
and $u_k v_k^*=\nabla\sigma_{\min}(X_k)$, any limit $G$ of $u_k v_k^*$ belongs to
$\partial^C\sigma_{\min}(\widehat X)$. Therefore $G\in\partial^C \widetilde{\sF}(\widehat\Delta)$
and $\Pis(G)$ is collinear with $\widehat\Delta$, proving the Clarke stationarity
statement.
\end{proof}

\begin{remark}
If $(u,v)$ is a solution to~\eqref{nlsys}, then
\begin{equation} \label{normalization ambiguity}
    (u\alpha^{-1}, v\overline{\alpha})
\end{equation} is another solution for each $\alpha\in\mathbb{C}, \alpha\neq 0$; hence the solutions are defined up to a normalization factor. We can assume without loss of generality that $\norm{v}=1$.
\end{remark}

\subsection{Differentials, gradients and Hessians}
We set
\[
G(u,v) = \begin{bmatrix}
    (A+\Delta)v\\
    (A+\Delta)^*u
\end{bmatrix}, \quad \Delta = \Pis(uv^*),
\]
so that~\eqref{nlsys} becomes $G(u,v) = 0$. It is interesting to note that, when $A\in\mathcal{S}$, this quantity $G(u,v)$ is the gradient of a scalar function $F(u,v)$. We prove it in the following lemma.

\begin{lemma}\label{lem:gradientAnotinS}
Consider the function $F(u,v) = \frac{1}{2}\norm{A+\Pis(uv^*)}_F^2$. Then, the two partial gradients of $F$ with respect to $u$ and $v$ are
    \begin{equation} \label{generalgradient}
        \nabla_u F(u,v) = (\Pis(A)+\Delta)v, \quad \nabla_v F(u,v) = (\Pis(A)+\Delta)^*u, \quad \Delta = \Pis(uv^*).
    \end{equation}
\end{lemma}
\begin{proof}
    Note that 
    \[
     \scalar{A+\Delta, A+\Delta} = \|A\|_F^2 + 2 \scalar{\Pis(A),uv^*}+\|\Pis(uv^*)\|_F^2.
    \]
    We compute the differential
    \begin{align*}
        dF(u,v) &= d \frac{1}{2}\left(\|A\|_F^2 + 2 \scalar{\Pis(A),uv^*}+\|\Pis(uv^*)\|_F^2\right)\\
        &= \scalar{\Pis(A), (du) v^* + u (dv)^*}+\scalar{\Pis((du) v^* + u (dv)^*),\Pis(uv^*)}\\
        &= \scalar{\Pis(A)+\Delta, (du) v^* + u (dv)^*}\\
        &= \scalar{(\Pis(A)+\Delta)v, du} + \scalar{dv, (\Pis(A)+\Delta)^*u},
    \end{align*}
    where we have used the definition of $\scalar{\cdot,\cdot}$, the fact that $\Pis$ is an orthogonal projection, and the cyclic property of the trace. From the last line we can read off the two gradients.
\end{proof}
When $A \in \mathcal{S}$, $\Pis(A) = A$ and hence~\eqref{generalgradient} coincide with the two blocks of $G(u,v)$. This result shows that~\eqref{nlsys} is a gradient system.

The relation between this $F(u,v)$ and the original minimization problem~\eqref{eq:distance} is not immediate; and in general, the global distance minimizer $\Delta_*$ is neither a minimum nor a maximum of the functional $F$: see Example~\ref{example 2x2} below. 

We can give an explicit formula for the differential $dG(u,v)$ of $G(u,v)$, which is a $\mathbb{R}$-linear operator $H: \mathbb{C}^{2n} \to \mathbb{C}^{2n}$. We use the letter $H$ because, in the case in which $G$ is the gradient of $F(u,v)$, the operator $H$ is the Hessian of $F(u,v)$.

We note that, even when $u,v$ are complex vectors, $H$ is only guaranteed to be $\mathbb{R}$-linear, since conjugates will appear in its expression. To fix the notation, let $\overline{x}$ denote the (entrywise) conjugate of a scalar, vector or matrix $x$;  while $x^\top$ stands for the transpose of $x$.

\begin{lemma} \label{lem:genhessian}
The differential of $G(u,v)$ is the $\mathbb{R}$-linear operator $H$ such that for each $\begin{bsmallmatrix}
    du\\dv
\end{bsmallmatrix} \in \mathbb{C}^{2n}$
\[
H \begin{bmatrix}
    du\\
    dv
\end{bmatrix} = 
\begin{bmatrix}
    MM^* \, du + (A+\Delta)dv + M N^\top\overline{dv}\\
    (A+\Delta)^*du + NM^\top \overline{du} + NN^*\,dv
\end{bmatrix},
\]
where $\mathcal{P}$ is defined as in \eqref{eq:matrix_P} and the matrices $M$ and $N$ are defined as follow:
\begin{align*}
M &= (v^\top \otimes I_n)\mathcal{P} = \begin{bmatrix}
    P^{(1)}v & P^{(2)}v & \dots & P^{(p)}v
\end{bmatrix} \in \mathbb{C}^{n\times p},\\
N &= (I \otimes u^\top)\overline{\mathcal{P}} = \begin{bmatrix}
    (P^{(1)})^* u & (P^{(2)})^* u & \dots & (P^{(p)})^* u
\end{bmatrix}\in \mathbb{C}^{n\times p}.
\end{align*}
\end{lemma}
\begin{proof}
We compute the Hessian by differentiating
    \begin{align*}
        d(A+\Delta)v &= \Pis(du\, v^*)v + \Pis(u (dv)^*)v + (A+\Delta)dv,\\
        d(A+\Delta)^*u &= \Pis(du\, v^*)^*u + \Pis(u\, (dv)^*)^*u + (A+\Delta)^*\, du.
    \end{align*}
We recall some identities already used in~\eqref{vecDelta}: for every $a,b \in \mathbb{C}^n$, we have
    \[
    \vvec \Pis(ab^*) = \mathcal{P}\mathcal{P^*}(\vvec ab^*) = \mathcal{P}\mathcal{P^*} (\overline{b}\otimes a) = \mathcal{P}\mathcal{P^*} (\overline{b}\otimes I_n)a = \mathcal{P}\mathcal{P^*} (I_n\otimes a)\overline{b}.
    \]
    These identities lets us simplify a few terms in $d(A+\Delta)v$ and $d(A+\Delta)^*u$, using:
    \begin{gather*}
        \Pis(du\, v^*)v = (v^\top \otimes I) \vvec \Pis(du\, v^*) = (v^\top \otimes I)\mathcal{P}\mathcal{P^*}(\overline{v}\otimes I_n)du = MM^*du,\\
         \Pis(u (dv)^*)v = (v^\top \otimes I) \mathcal{P}\mathcal{P}^*(I_n\otimes u)\overline{dv} = MN^\top \overline{dv},\\
        \Pis(u\, (dv)^*)^*u = \vvec(u^\top \overline{\Pis(u(dv)^*)}) = (I_n \otimes u^\top)\overline{\mathcal{P}\mathcal{P^*} (I_n\otimes u)\overline{dv}}\\ = (I_n \otimes u^\top)\overline{\mathcal{P}} \mathcal{P}^\top(I_n\otimes \overline{u})dv = NN^*dv,\\
        \Pis(du\, v^*)^*u = \vvec (u^\top \overline{\Pis(du\, v^*)}) = (I_n\otimes u^\top)\overline{\mathcal{P}\mathcal{P}^*(\overline{v}\otimes I_n) du}\\ = (I_n\otimes u^\top) \overline{\mathcal{P}}\mathcal{P}^\top (v\otimes I_n) \overline{du} = NM^\top \overline{du}.
    \end{gather*}
\end{proof}

Moreover, for the special case of the projection on a real sparsity structure, this expression can be further simplified. 
\begin{lemma}\label{lem:Hessian}
    Let $\mathcal{S}$ be the subspace of matrices with sparsity structure $\mathcal{J}$, i.e., $B_{ij}=0$ if $(i,j)\not\in\mathcal{J}$. Assume $A\in\mathcal{S}$ and that $A,u,v$ have real entries; then, the operator $H(u,v)$ (over real vectors) is
    \begin{equation} \label{Hessian of F}
        H = \begin{bmatrix}
            K_1 & A+2\Delta\\
            (A+2\Delta)^* & K_2
        \end{bmatrix},
    \end{equation}
    where we have set again $\Delta = \Pi_\mathcal{S}(uv^*)$, and $K_1, K_2$ are the two diagonal matrices such that
    \begin{align*}
        (K_1)_{ii} = \sum_{j \text{ s.t. }(i,j)\in\mathcal{J}} v_j^2, \quad (K_2)_{jj} = \sum_{i \text{ s.t. }(i,j)\in\mathcal{J}} u_i^2.
    \end{align*}
\end{lemma}
\begin{proof}
    We may take $(e_ie_j^\top \colon (i,j)\in\mathcal{J})$ as an orthonormal basis of $\mathcal{S}$. Then, the columns of $M$ are $e_ie_j^*v = e_i v_j$, and the columns of $N$ are $e_je_i^*u = e_j u_i$. It follows that in Lemma~\ref{lem:genhessian}
    \begin{align*}
    MM^* &= MM^\top = \sum_{(i,j)\in\mathcal{J}} e_i v_j v_j e_i^\top = K_1,\\
    NN^* &= NN^\top = \sum_{(i,j)\in\mathcal{J}} e_j u_i u_i e_j^\top = K_2,\\
    MN^\top &= \sum_{(i,j)\in\mathcal{J}} e_i v_ju_i e_j^\top = \Delta. 
    \end{align*}
\end{proof}

\begin{example} \label{example 2x2}
    Let $A= \operatorname{diag}(\sigma_1,\sigma_2)$, with $\sigma_1 > \sigma_2 > 0$, and $\mathcal{S} = \mathbb{R}^{2\times 2}$ be the trivial sparsity structure satisfied by every $2\times 2$ matrix ($\mathcal{J}=\{(1,1),(1,2),(2,1),(2,2)\}$, $\Pis = I$). Then, the closest (structured) singular matrix to $A$ is $\operatorname{diag}(\sigma_1,0)$, by the Eckhart-Young theorem, corresponding to the rank-1 perturbation $\Delta = uv^*$ with $u=- e_2 \sigma_2$, $v=e_2$. We can verify that $(A+\Delta)v=(A+\Delta)^*u=0$ holds, i.e., \eqref{nlsys}~is satisfied.
    
    Moreover, in this point $(u,v) = (-\sigma e_2, e_2)$, \eqref{Hessian of F}~gives
    \[
    H = \nabla^2_{[u,v]}F = \begin{bmatrix}
    1 & 0 & \sigma_1 & 0\\
    0 & 1 & 0 & -\sigma_2\\
    \sigma_1 & 0 & \sigma_2^2 & 0\\
    0 & -\sigma_2 & 0 & \sigma_2^2
    \end{bmatrix}.
    \]
    The eigenvalues of $H$ are the union of those of
    \[
    H_1 = 
    \begin{bmatrix}
        1 & \sigma_1\\
        \sigma_1 & \sigma_2^2
    \end{bmatrix}, \quad 
    H_2 = \begin{bmatrix}
        1 & -\sigma_2 \\ -\sigma_2 & \sigma_2^2
    \end{bmatrix}.
    \]
    The matrix $H_1$ is indefinite, since $\sigma_1 > \sigma_2$, and hence it has a positive and a negative eigenvalue. The matrix $H_2$ is semidefinite, with eigenvalues $0$ and $1+\sigma_2^2$. The presence of this zero eigenvalue reflects the fact that the solution is overparametrized: $F(u,v) = F(\alpha u, \frac{1}{\alpha}v)$, and hence $\frac{d}{d\alpha} F(\alpha u, \frac{1}{\alpha}v) = 0$.

    As the Hessian $H$ is indefinite, the solution of the distance to singularity problem is a saddle point of $F(u,v)$, not a local minimum or maximum.
\end{example}

\section{An algorithm based on the Newton method}
\label{sec:Newton}

The nonlinear system \eqref{nlsys} can be solved with a Newton approach. There are several issues to discuss.

\subsection{Fixing the normalization}
Recall that the solutions of~\eqref{nlsys} are defined up to a normalization factor as in~\eqref{normalization ambiguity}; and, consequently, the differential $H$ is singular. In order to avoid a spurious degree of freedom, we wish to compute only solutions with $\norm{v}=1$.  To devise a strategy to enforce this normalization, we first reason on the case in which $G(u,v)$ is a gradient system. Then, we can choose $\beta > 0 $ and define
\begin{equation*}
F_{\beta}(u,v) = \frac{1}{2} \norm{A + \Pi_{\mathcal{S}}(uv^*)}_F^2 + \frac{\beta}{4} \left( \norm{v}^2-1 \right)^2.
\end{equation*}
The gradient of $F_{\beta}(u,v)$ is
\[
G_{\beta}(u,v) =
    \begin{bmatrix}
    (A+\Delta)v\\
    (A+\Delta)^*u + \beta(\norm{v}^2-1)v
\end{bmatrix}, \quad \Delta = \Pis(uv^*).
\]
Clearly, if $u,v$ are a stationary point for $F(u,v)$ and $\norm{v}=1$ then they are also a stationary point for $F_{\beta}(u,v)$, since the gradients of both summands are zero.

Even when $A \not \in \mathcal{S}$ and $G(u,v)$ is not a gradient system, we can still formulate the modified equation $G_{\beta}(u,v)=0$: indeed, we can verify directly that any solution $(u,v)$ to $G(u,v)=0$ in which $\norm{v}=1$ is also a solution to $G_{\beta}(u,v)$.

The differential of $G_{\beta}(u,v)$ is
\[
H_\beta(u,v) = H(u,v) + \begin{bmatrix}
    0 & 0\\
    0 & 2\beta vv^* + \beta(\norm{v}^2-1)I 
\end{bmatrix},
\]
where $H(u,v)$ is defined in Lemma \ref{lem:genhessian}. With this differential, we write down the multivariate Newton method for finding a solution $(u,v)$ to the nonlinear system of equations $G_{\beta}(u,v) = 0$: \begin{equation} \label{Newton k+1 formula}
\begin{bmatrix}
u_{k+1}\\ v_{k+1}
\end{bmatrix} = \begin{bmatrix}
u_{k}\\ v_{k}
\end{bmatrix} + \begin{bmatrix}
\delta_{u_{k}} \\ \delta_{v_{k}}
\end{bmatrix}
\end{equation}
where
\begin{equation} \label{Newton increment formula}
\begin{bmatrix}
\delta_{u_{k}} \\ \delta_{v_{k}}
\end{bmatrix} = - H_{\beta}^{-1}(u,v) G_{\beta}(u,v).
\end{equation}

\subsection{Line search}
We modify the update formula~\eqref{Newton increment formula} to obtain an algorithm with line search, which we formulate as Algorithm~\ref{algo:main}.

\begin{algorithm}
    \KwData{nonsingular matrix $A$, structure $\mathcal{S}$}
    \KwResult{A solution of $G(u,v)=0$, $\norm{v}=1$, and the associated $\Delta\in\mathcal{S}$ such that $A+\Delta$ is singular (hopefully, a global minimizer)}
    Choose $\beta>0$ (e.g., $\beta=\norm{A}_F$)\;
    $u,v \gets \text{minimum left and right singular value of $A$}$\;
    \While{convergence}{
        $(\delta_u,\delta_v) \gets \text{Newton increment in~\eqref{Newton increment formula}}$\;
        $\alpha \gets 1$\tcp*{candidate step size}
        \While{$\norm{G_{\beta}(u+\alpha\delta_u, v+\alpha\delta_v)} \geq \norm{G_{\beta}(u,v)}$}{
            $\alpha \gets \alpha/2$\;
        }
        $(u,v) \gets (u+\alpha\delta_u, v+\alpha\delta_v)$\tcp*{ensures decrease of $\norm{G_{\beta}(u,v)}$}
    }
    $\Delta \gets \Pis(uv^*)$\;
    \caption{Newton method with backtracking line search} \label{algo:main}
\end{algorithm}

In words, if the Newton step~\eqref{Newton increment formula} would produce a new iterate $(u_{k+1},v_{k+1}) = (u_k + \delta_{u_k}, v_k + \delta_{v_k})$ that does not reduce the value of $\norm{G_\beta(u_k,v_k)}$, then we reduce the step-length by half and test a new candidate $u_{k+1} = u_k + \frac{1}{2}\delta_{u_k}, v_{k+1} = v_k + \frac{1}{2}\delta_{v_k}$; we continue in this fashion halving the step length until the norm is reduced. If $H_\beta(u,v)$ is positive definite, the Newton direction is always a descent direction, and this strategy should ensure an eventual decrease.

This backtracking strategy is a very simple form of \emph{globalization} of Newton's method, to ensure that we produce a decreasing sequence $\norm{G_\beta(u_k,v_k)}$. We refer the reader to~\cite[Chapter~3]{NocedalWright} for a broader discussion of line search methods in optimization and more sophisticated strategies.

In practice, we observed that in most cases the choice $\alpha=1$ (the classical Newton method without line search) is sufficient to get a reduction of the norm, without the need for backtracking steps; but this safeguard increases the robustness of the algorithm with minimal additional cost.

\subsection{Starting values}
\label{sec:starteps}

Since the behavior of the Newton method depends on the choice of the initial point, we describe an effective strategy for selecting initial choices for the vectors $u,v$ in Algorithm~\ref{algo:main}. 

We recall that in the unstructured case the minimum of~\eqref{eq:distance} is given by $\Delta = - \sigma_n u_n v_n^*$, where $u_n,v_n$ are the left and right singular vectors associated with the smallest singular value $\sigma_n$ of the matrix $A$. We may use this unstructured minimizer as the starting point of the structured problem, setting
$
    u_{\text{initial}} = -\sigma_n u_n, \, v_{\text{initial}} = v_n,$
but it turns out that in general there is a better choice than $\sigma_n$ for the scaling coefficient. In the following, we describe a method to choose $\sigma$ in an initial value of the form
\begin{equation}
\label{eq:initial}
u_{\text{initial}} = -\sigma u_n, \quad v_{\text{initial}} =  v_n,
\end{equation}
inspired by the ideas in Subsection \ref{subsec:gs}. 

According to Subsection \ref{subsec:gs}, we consider --- as a functional to minimize wrt $E$ (of unit Frobenius norm) --- the following:
\begin{equation} \nonumber
\widetilde{F}_{\eps}(E) = \sigma_{n} (A + \eps E), \qquad \mbox{where} \qquad
\sigma_{n} (A + \eps E) = \sigma_{\min} (A + \eps E).
\end{equation}

Recall that $\Delta = \Pis(u v^*) = \eps E$, where $\norm{E}_F=1$. 
Let $\widetilde{G}$ be the  gradient of $E \mapsto \widetilde{F}_{\eps}(E)$
at $\eps=0$, i.e.,
$\widetilde{G}= \Pis (u_n v_n^*) $ with $u_n$ and $v_n$ the left and right singular vectors associated to the smallest singular value $\sigma_{n} (A)$, which we suppose to be simple and nonzero. 

We wish to approximate the smallest solution of
\begin{equation*}
\varphi(\eps):=\widetilde{F}_\eps(E(\eps))=0, 
\end{equation*}
where $E(\eps)$ indicates  a stationary point of the gradient system \eqref{eq:odeEstru}, and thus a (local) minimizer of $\widetilde{F}_\eps(E)$ - for given $\eps \ge 0$ - over set of matrices of unit Frobenius norm. 
Clearly at $\eps=0$, $E$ does not play any role and $\varphi(0) = \sigma_n(A)$.
In order to compute an effective starting value $\eps=\eps_0$, we formally apply a Newton step to equation \eqref{eq:zero}, starting from the value $\eps=\eps_{-1} = 0$.

For $k=-1$ with $\eps_{-1}=0$, the Newton iteration gives
\begin{equation} \label{eq:eps0F}
\eps_{\rm initial} :=  \eps_0 = {0}-\frac{\varphi(0)}{\varphi'(0)} 
= {}-\frac{\sigma_{n}(A)}{\varphi'(0)}.
\end{equation}
Assuming the smoothness of $E(\eps)$ wrt $\eps$, 
we use \cite[Theorem IV.1.4]{GL25} to express the derivative of $\varphi(\eps)$ in explicit form (for $\eps$ smaller than the structured distance to singularity, i.e. s.t $\varphi(\eps) > 0$), namely 
 
\begin{equation*}
\varphi'(\eps) = -\| \widetilde{G}(\eps) \|_F = {}-\| \Pis \left(\hat{u}(\eps) v(\eps)^* \right) \|_F, 
\end{equation*}
where we recall that $\hat{u} = u/\| u\|$. 
Therefore at $\eps=0$, where $\hat{u}=u_n$ and $v=v_n$, we have
\begin{equation} \nonumber
\varphi'(0) = {}-\| \Pis \left( u_n v_n^* \right) \|_F.
\end{equation}
As a consequence \eqref{eq:eps0F}
yields
\begin{equation} \label{eq:eps0F2}
\eps_{0} = \frac{\sigma_{n}(A)}{\| \Pis \left( u_n v_n^* \right) \|_F};
\end{equation}
we use this value of $\varepsilon_0$ as the norm for the initial perturbation
$\Delta$. Hence, we choose $\sigma$ in~\eqref{eq:initial} so that $\norm{\Delta} = \varepsilon_0$:
\begin{equation} \label{eq:choice-sigma}
\sigma= \frac{\eps_0}{\| \Pis \left( u_n v_n^* \right) \|_F} =   
\frac{\sigma_n}{\| \Pis \left( u_n v_n^* \right) \|^2_F}.
\end{equation}

\begin{remark}
    The starting value heuristic $\sigma$ in~\eqref{eq:choice-sigma} can be obtained also from a different argument: in an initial value of the form~\eqref{eq:initial}, we select the value of $\sigma$ that makes $A+\Delta$ orthogonal (in the Frobenius scalar product) to $u_n v_n^*$. Indeed, if we plug $\Delta = \Pis(u_{\text{initial}} v_{\text{initial}}^*) = -\sigma\Pis(u_n v_n^*)$ into $\langle A+\Delta, u_nv_n^* \rangle=0$ and solve for $\sigma$, we get
    \[
    \sigma = \frac{u_n^*Av_n}{\langle \Pis(u_nv_n^*), u_nv_n^*\rangle} = \frac{\sigma_n}{\norm{\Pis(u_nv_n^*)}_F^2},
    \]
    using the fact that $\Pis$ is an orthogonal projection. Assuming $A \in \mathcal{S}$, the choice of $\sigma$ can be also interpreted geometrically as follows:
we select $\sigma$ so that $A+\Delta$ is orthogonal 
to the structured steepest descent direction (structured negative gradient) $\Pi_S(u_n v_n^*)$, i.e.,
\[
\langle A+\Delta,\; \Pis(u_n v_n^*) \rangle = 0.
\]
Since $A \in \mathcal{S}$ and $\Pis$ is the orthogonal projector onto $\mathcal{S}$, we have
$
\langle A,\Pis(u_n v_n^*)\rangle
=
\langle A,u_n v_n^* \rangle
=
\sigma_n. 
$
Therefore, the initial perturbation is obtained by canceling the component of $A$
along the gradient direction $\Pis(u_n v_n^*)$.
\end{remark}

\subsubsection*{Multiple initial values}
While the above choice of starting values for $u$ and $v$ is reasonable when only a single starting point is used, in some problems it might be necessary to run test several starting values to reduce the risk of getting trapped in a local minimum.
Indeed, there may be examples where the nonsingular matrix $A$ has a set of singular values of small and comparable magnitude. In such settings, adding structure may give an optimization problem~\eqref{eq:distance} with multiple local minima of comparable magnitude; and the initial choice in \eqref{eq:initial}--\eqref{eq:choice-sigma} does not guarantee convergence to the global optimum. Therefore, we suggest running the method in Algorithm \ref{algo:main} for several choices of the initial vectors $u_{\text{initial}}$, and $v_{\text{initial}}$. In detail, we select $K$ pairs of left and right singular vectors $u_{n-k+1}$ and $v_{n-k+1}$ associated with the $n-k+1$-th smallest singular value $\sigma_{n-K+1}$, $k=1,\ldots,K$, and run the method $K$ times, choosing (for each $k$)
\begin{equation*}
u_{\text{initial}} = \widehat{\sigma}_{n-k+1} u_{n-k+1}, \quad v_{\text{initial}} =  v_{n-k+1},
\end{equation*}
with
\begin{equation} \label{eeq:choice}
\widehat\sigma_{n-k+1} = 
\frac{\sigma_{n-k+1}}{\norm*{\Pis \left( u_{n-k+1} v_{n-k+1}^* \right)}^2_F}
\end{equation}
The computed final perturbation $\Delta_* \in \cS$, determining the singularity of $A + \Delta_*$, is then selected as the one of smallest Frobenius norm among the $K$ computed ones.

An example where the solution benefits from the multiple initializations is presented in the subsequent Section~\ref{sec:example multiple}.

As a cheaper alternative to the one described above, one may select
\begin{equation*}
u_{\text{initial}} = \widehat{\sigma}_{n-\ell+1} u_{n-\ell+1}, \quad v_{\text{initial}} =  v_{n-\ell+1}, \qquad \mbox{with} \qquad \ell = \argmin_{1 \le k \le K}
\widehat{\sigma}_{n-k+1},
\end{equation*}
which would avoid applying the proposed method several times.

\section{Numerical experiments}
\label{sec:numerical exp}

In this section, we report several numerical experiments that were performed using the Matlab implementation of our method available on~\url{https://github.com/fph/NearestSingularAsSystem/}. The timings reported refer to a laptop with an Intel Core i5-1135G7 and Matlab R2025b.

\subsection{A large-scale matrix}
To illustrate the numerical efficiency of the new method, we take an experiment that appears in~\cite{guglielmi2023rank, oracle}: finding the nearest singular matrix to the matrix \texttt{orani678} in the SuiteSparse Matrix collection, preserving the same sparsity pattern structure. This is an real unsymmetric $2529\times 2529$ sparse matrix with $90158$ nonzero elements; it is a case where the minimizer occurs in a rank-drop point for $M(v)$ in the Riemann-Oracle method, so regularization is needed. We run Algorithm~\ref{algo:main} for this problem, taking advantage of the sparseness: we use \texttt{svds} to compute the left and right singular vectors associated to the smallest singular value, as a starting point, and we use \texttt{minres} with a very loose tolerance of $10^{-2}$ to solve the system~\eqref{Newton increment formula}. 

The method takes 5 iterations (with an average of 2083.6 matrix-vector products per iteration inside \texttt{minres}), and converges to a matrix $A+\Delta$ with $\sigma_{\min}(A+\Delta) \approx 1.5481\times 10^{-13}$ and $\sigma_{\max}(A+\Delta) \approx 3.20\times 10^1$. The algorithm takes less than 3 seconds: this time compares very favorably with the methods in~\cite{oracle} and~\cite{guglielmi2023rank}, which take more than 30 seconds to solve the problem on the same test machine. The computed minima coincide up to at least 7 significant digits.

\subsection{A case requiring multiple starting values}
\label{sec:example multiple}

We test the algorithm on the matrix $C$ available at \url{https://github.com/fph/NearestSingularAsSystem/blob/main/example_starting_value.mat}. This is a nonsymmetric $50\times 50$ matrix with 50\% density of nonzeros; it has been obtained as the sparsification of a random-generated orthogonal matrix. This specific matrix has been chosen because the smallest singular value and vectors $\sigma_n,u_n,v_n$ are not the ones that produce the smallest $\varepsilon_0$ in~\eqref{eq:eps0F2}. We report in Table~\ref{tab:initialvalues} the value of the 5 smallest singular values of $C$, and the magnitude of the local minima $\norm{\Delta_*}_F$ obtained using them as starting values.
\begin{table}[h]
    \centering
    \begin{tabular}{c@{\hspace{20pt}}c@{\hspace{20pt}}c}
    \toprule
        $k$ & $\sigma_k$ & $\norm{\Delta_*}_F$ \\
        \midrule
        46 &     0.1552 &         0.2321\\
        47 &     0.1015 &     0.1489\\
        48 &      0.0574 &      0.0793 \\
        49 &     0.0401 &     \textbf{0.0571}\\
        50 &     0.0389 &     0.0639\\
        \bottomrule
    \end{tabular}
    \caption{Values of $\|\Delta_*\|_F$ for the solutions to $G_\beta(u,v)=0$ obtained constructing starting values as in Section~\ref{sec:starteps} from the smallest 5 singular values and vectors of $C$.}
    \label{tab:initialvalues}
\end{table}
We see that the smallest norm is obtained with the second-smallest singular value, not the smallest. This example shows that using multiple starting values can produce better solutions.

Since the initial matrix is small, the computation was performed using direct $O(n^3)$ methods to compute the SVD and to solve the linear system in~\eqref{Newton increment formula}. The run time of Algorithm~\eqref{algo:main}, repeated 5 times with 5 different starting values, is less than 0.02 seconds.

\subsection{An example from polynomial $\varepsilon$-GCD computation}
We test the new method also on another example that appears in~\cite{oracle,UseM17gcd}: given the polynomials
\begin{equation} \label{boito811}
p(x) = \gamma_p\prod_{j=1}^{10} (x-\alpha_j), \quad q(x) = \gamma_q\prod_{j=1}^{10} (x-\alpha_j + 10^{-j}), \quad \alpha_j = (-1)^j\tfrac{j}{2},
\end{equation}
we look for the smallest perturbation $\tilde{p},\tilde{q}$ that gives a pair of polynomials with GCD of prescribed degree $d$. The normalization coefficients $\gamma_p$ and $\gamma_q$ appearing in~\eqref{boito811} are chosen so that the vectors of coefficients of $p$ and $q$ have Euclidean norm 1.

This problem can be formulated as the distance to singularity problem for a Sylvester matrix $A_S$  defined as 
\begin{equation}
    A_S = \begin{bmatrix}
        \frac{1}{\sqrt{\deg(q)-d+1}} \mathcal{T}_{p} & 
        \frac{1}{\sqrt{\deg(p)-d+1}} \mathcal{T}_{q}
    \end{bmatrix},
\end{equation}
composed of two blocks with Toeplitz structure
\begin{equation*}
    \mathcal{T}_{p} = \underbrace{\begin{bmatrix}
    p_0\\
    p_1 & p_0\\
    \vdots & p_1 & \ddots\\
    p_{10} & \vdots & \ddots & p_0\\
        & p_{10} & \ddots & p_1\\
        &     & \ddots & \vdots\\
        &     &        & p_{10}\\
\end{bmatrix}}_{\text{$\deg(q)-d+1$ columns}}, \quad 
\mathcal{T}_{q} = \underbrace{\begin{bmatrix}
    q_0\\
    q_1 & q_0\\
    \vdots & q_1 & \ddots\\
    q_{10} & \vdots & \ddots & q_0\\
        & q_{10} & \ddots & q_1\\
        &     & \ddots & \vdots\\
        &     &        & q_{10}\\
\end{bmatrix}}_{\text{$\deg(p)-d+1$ columns}}.
\end{equation*}
We refer to~\cite[Section~7]{oracle} for more details. The problem is quite challenging numerically, as the objective function $\norm{\Delta}_F^2$ is smaller than the machine precision $\approx 10^{-16}$ for $d \leq 5$.

We report the results obtained by the new method, and a comparison with~\cite{oracle,UseM17gcd}, in Table~\ref{tbl:gcd}.
\begin{table}[]
    \centering
        \begin{tabular}[t]{c@{\hspace{10pt}}c@{\hspace{10pt}}c@{\hspace{10pt}}c@{\hspace{10pt}}c@{\hspace{10pt}}c}
    \toprule
    \multicolumn{6}{c}{Approximate GCDs of~\eqref{boito811}}\\
    $d$ & Algorithm~\ref{algo:main} & Riemann-Oracle & $VP_S$~\cite{UseM17gcd} & $VP_g$, $VP_h$~\cite{UseM17gcd} & UVGCD\\
    \midrule
        9 & $3.9964\cdot 10^{-3}$ & $3.9964\cdot 10^{-3}$ & $4.00\cdot10^{-3}$ & $4.00\cdot10^{-3}$ & $3.996\cdot 10^{-3}$\\
        8 & $1.7288\cdot 10^{-4}$ & $1.7288\cdot 10^{-4}$ & $1.73\cdot 10^{-4}$ & $1.73\cdot 10^{-4}$ & $1.729\cdot 10^{-4}$\\
        7 & $7.0890\cdot 10^{-6}$ & $7.0890\cdot 10^{-6}$ & $\underline{3.93\cdot 10^{-4}}$ & $7.09\cdot 10^{-6}$ & $7.089\cdot 10^{-6}$\\
        6 & $1.8293\cdot 10^{-7}$ & $1.8293\cdot 10^{-7}$ & $\underline{2.7\cdot 10^{-5}}$ & $1.83\cdot 10^{-7}$ & $1.829\cdot 10^{-7}$\\
        5 & $\underline{2.0201}\cdot 10^{-9}$\makebox[0pt][l]{(*)} & $4.4\underline{913}\cdot 10^{-9}$ & $\underline{1.72\cdot 10^{-5}}$ & $4.49\cdot 10^{-9}$ & $4.487\cdot 10^{-9}$\\
        4 & $\underline{2.7776}\cdot 10^{-11}$ \makebox[0pt][l]{(*)} & \underline{$1.8293\cdot 10^{-7}$} & $\underline{1.55\cdot 10^{-14}}$ & $8.40\cdot 10^{-11}$ & $8.40\cdot 10^{-11}$\\
    \bottomrule
    \end{tabular}
    \caption{Computed distance $\norm{(p-gu, q-gw)}$ of the pair of polynomials $p,q$ from a pair having a gcd $g$ of prescribed degree $d$, and comparison with~\cite{oracle,UseM17gcd} and with the results of UVGCD from~\cite{nagasaka}. Not all sources report the same number of significant digits. The results labeled with an asterisk (*) in the first column are cases in which the obtained $\Delta$ could not be used to construct explicitly a triple of polynomials $g,u,w$; hence signifying a numerical failure. Underlined digits are those which differ from the results of UVGCD, which we consider to be the most reliable.} \label{tbl:gcd}
\end{table}
The results show that the method obtains reliable results up to the point where $\norm{\Delta}_F^2$ becomes smaller than the machine precision. The comparison with the results in~\cite[Table~1]{UseM17gcd} is interesting: the new method performs better than SLRA on the same matrix formulation $A_S$ of the polynomial GCD problem (column $VP_S$ in~\cite[Table~1]{UseM17gcd}), even though SLRA obtains better results when run on a different matrix formulation of the problem (columns $VP_g$, $VP_h$) that does not fall in the exact same framework described in Section~\ref{sec:slra}. The results of $VP_g$, $VP_h$ from~\cite{UseM17gcd} match those of a state-of-the-art specialized algorithm (column $UVGCD$, reported with more significant digits from~\cite{nagasaka}). Overall, the results suggest that the new solution algorithm is reliable up to the point where $\norm{\Delta}_F^2$ becomes smaller than the machine precision: it is as accurate as the one in~\cite{oracle} and slightly more than the one in~\cite{UseM17gcd}, when run on the same matrix $A_S$. 

\section{Conclusions}
\label{sec:conclusions}

We have introduced and studied a new method to numerically approximate the structured distance to singularity. After a theoretical comparison between two existing methods, we propose a novel approach, which is independent of the existing ones. The technique solves the nearness problem via a Newton method on a system of nonlinear equations, avoiding the nested optimizations arising in \cite{oracle} and \cite{guglielmi2023rank}. We applied this technique both to large, sparse matrices and smaller problems with Toeplitz-like structures. Further research directions include expanding the theoretical analysis of the proposed method, and generalizing it to distance to stability problems.

\paragraph{Acknowledgements}
     All the authors are affiliated to the Italian INdAM-GNCS (Gruppo Nazionale di Calcolo Scientifico). FP acknowledges the support by the National Centre for HPC, Big Data and Quantum Computing–HPC, CUP B83C22002940006, funded by the European Union--NextGenerationEU, and by the Italian Ministry of University and Research through the PRIN project 2022 ``MOLE: Manifold constrained Optimization and LEarning'', CUP B53C24006410006. SS acknowledges the support by the European Union (ERC consolidator, eLinoR, no 101085607). NG acknowledges that his research was supported by funds from the Italian 
     MUR (Ministero dell'Universit\`a e della Ricerca) within the PRIN 2022 Project ``Advanced numerical methods for time dependent parametric partial differential equations with applications'' and the PRIN-PNRR Project ``FIN4GEO''. He also acknowledges funding from the Dipartimento di Eccellenza 2023–2027 project awarded to the Gran Sasso Science Institute (GSSI) by the Italian Ministry of University and Research (MUR).

\bibliographystyle{abbrv}
\addcontentsline{toc}{section}{Bibliography}
\bibliography{biblio}

\end{document}